\documentclass[12pt,a4paper,reqno]{amsart}
\usepackage{amsmath,amsfonts,amssymb,amsthm,xcolor,amscd,tikz,xspace,enumerate}
\usepackage[mathscr]{eucal}
\usepackage[pagebackref, colorlinks=true,linkcolor=blue,citecolor=red]{hyperref}
\usepackage{pdfpages}
\usetikzlibrary{arrows}
\usetikzlibrary{arrows.meta}
\usetikzlibrary{intersections}

\usetikzlibrary{patterns}
\usetikzlibrary{shadings,intersections,calc}

 \newtheorem{thm}{Theorem}[section]
 \newtheorem{cor}[thm]{Corollary}
 \newtheorem{lem}[thm]{Lemma}
 \newtheorem{prop}[thm]{Proposition}
 \theoremstyle{definition}
 \newtheorem{defn}[thm]{Definition}
 \newtheorem{problem}[thm]{Problem}

 \newtheorem{rem}[thm]{Remark}
 \newtheorem{ex}[thm]{Example}
 \numberwithin{equation}{section}

\numberwithin{equation}{section}

\def\.{{\cdot}}

\def\<{\langle} \def\>{\rangle}

\begin{document}

\title[Factorization number and subgroup commutativity degree...]
 {Factorization number and subgroup commutativity degree 
 via spectral invariants}


\author[S.K. Muhie]{Seid Kassaw Muhie}
\address{Seid Kassaw Muhie  \endgraf
Department of Mathematics \endgraf
Woldia University, Woldia, Ethiopia\endgraf 
}
\email{seidkassaw063@gmail.com \endgraf} 

\author[D.E. Otera]{Daniele Ettore Otera}
\address{Daniele Ettore Otera \endgraf  
Institute of Data Science and Digital Technologies\endgraf 
 Vilnius University, Vilnius, Lithuania\endgraf
  }
\email{daniele.otera@mif.vu.lt}

\author[F. G. Russo]{Francesco G. Russo}
\address{ Francesco G. Russo \endgraf
Department of Mathematics and Applied Mathematics\endgraf
University of Cape Town, Cape Town, South Africa\endgraf 
}
\email{francescog.russo@yahoo.com \endgraf} 




\date{17th of December 2022}

\begin{abstract}

The factorization number $F_2(G)$ of a finite group $G$ is the number of all possible factorizations of $G=HK$ as product of its subgroups $H$ and $K$, while the subgroup commutativity degree $\mathrm{sd}(G)$ of $G$ is the probability of finding two commuting subgroups in $G$ at random. It is known that $\mathrm{sd}(G)$ can be expressed in terms of $F_2(G)$. Denoting by $\mathrm{L}(G)$ the subgroups lattice of $G$, the non--permutability graph  of subgroups  $\Gamma_{\mathrm{L}(G)}$  of $G$ is the graph with vertices in $\mathrm{L}(G) \setminus \mathfrak{C}_{\mathrm{L}(G)}(\mathrm{L}(G))$, where $\mathfrak{C}_{\mathrm{L}(G)}(\mathrm{L}(G))$ is the smallest sublattice of $\mathrm{L}(G)$ containing all permutable subgroups of $G$, and edges obtained by joining two vertices $X,Y$   such that $XY\neq YX$. The spectral properties of $\Gamma_{\mathrm{L}(G)}$ have been recently investigated in connection with $F_2(G)$ and  $\mathrm{sd}(G)$. Here we show a  new combinatorial formula, which allows us to express  $F_2(G)$, and so $\mathrm{sd}(G)$, in terms of  adjacency and Laplacian matrices of $\Gamma_{\mathrm{L}(G)}$. 
 \end{abstract}

\maketitle


\section{Introduction  and statement of the main result}

 In the present paper we shall be interested  only in finite groups.   The  \textit{non-permutability graph of subgroups} $\Gamma_{\mathrm{L}(G)}$  of a group $G$ is the undirected and unweighted simple graph 
 defined as the ordered pair of vertices and edges  
  \begin{equation}\label{npgs}
\Gamma_{\mathrm{L}(G)}=(V(\Gamma_{\mathrm{L}(G)}) , E(\Gamma_{\mathrm{L}(G)})), 
\end{equation}
where  $\mathrm{L}(G)$ denotes the lattice of subgroups of $G$, \begin{equation}\label{vertices}
V(\Gamma_{\mathrm{L}(G)})=\mathrm{L}(G) \setminus \mathfrak{C}_{\mathrm{L}(G)}\big(\mathrm{L}(G)\big),
\end{equation}
  \begin{equation}\label{edge}
E(\Gamma_{\mathrm{L}(G)})=\{(X,Y) \in V(\Gamma_{\mathrm{L}(G)}) \times V(\Gamma_{\mathrm{L}(G)}) \mid X \sim Y \ \Longleftrightarrow \  XY\neq YX\}, 
\end{equation}
 and $\mathfrak{C}_{\mathrm{L}(G)}(X)$  is the set of all subgroups of $\mathrm{L}(G)$ commuting with $X \in \mathrm{L}(G)$. In other words 
  \begin{equation}\label{centralizer}
\mathfrak{C}_{\mathrm{L}(G)}(X)=\{Y\in \mathrm{L}(G) \ | \ XY=YX\}.
\end{equation}
Since the intersection 
  \begin{equation}\label{intersection}
{\underset{X \in \mathrm{L}(G)} \bigcap \mathfrak{C}_{\mathrm{L}(G)}(X)}=\{Y \in \mathrm{L}(G) \ | \ YX=XY, \  \ \ \forall X \in \mathrm{L}(G)\} 
\end{equation} 
is not (in general)  a sublattice of $\mathrm{L}(G)$,   we will consider  the smallest sublattice of $\mathrm{L}(G)$ containing \eqref{intersection}. This is denoted by  $\mathfrak{C}_{\mathrm{L}(G)}(\mathrm{L}(G))$ and appears in \eqref{vertices} above.

The non-permutability graph of subgroups is motivated by a line of research in lattice theory, which has analogies with the contributions \cite{ dr,  drtris, sr2}, where combinatorial properties of graphs and groups are discussed. 

In our present work we shall also use some spectral properties and invariants  of graphs in order to get  information on algebraic properties of corresponding groups.

The adjacency matrix  of $\Gamma_{\mathrm{L}(G)}$ is the square matrix \begin{equation}\label{lm}
    A(\Gamma_{\mathrm{L}(G)} ) = {(a_{X,Y})}_{X,Y \in V(\Gamma_{\mathrm{L}(G)})}, \ \   \ \ 
 \mbox{where} \ \
  a_{X,Y} =\left\{\begin{array}{lcl}  \ 1, & \mbox{if} \ (X,Y)\in E(\Gamma_{\mathrm{L}(G)}),\\
\ 0, &  \mbox{if} \ (X,Y)\not\in E(\Gamma_{\mathrm{L}(G)})
.\end{array}\right.
\end{equation}

Note that the degree of a vertex $X$ in \eqref{npgs} is defined by
\begin{equation}\label{degree}
\mathrm{deg}(X) = \sum_{Y\in V(\Gamma_{\mathrm{L}(G)})} a_{X,Y} . 
\end{equation}

Since $\Gamma_{\mathrm{L}(G)}$ is an undirected graph without loops, the Laplace matrix of $\Gamma_{\mathrm{L}(G)}$ is the matrix \begin{equation}\label{lapl}L(\Gamma_{\mathrm{L}(G)})= D-A(\Gamma_{\mathrm{L}(G)}), \end{equation} 

where $ D=\mathrm{diag}(\mathrm{deg}(X_i))$, 
for all $X_i\in V(\Gamma_{\mathrm{L}(G)})$ and $i=1, 2, \cdots, m=|V(\Gamma_{\mathrm{L}(G)})|$. These are common notions, which are usually considered in spectral graph theory, see  \cite {cd09, chung}.

 On the other hand, we are also interested in the so-called \textit{subgroup commutativity  degree} of   $G$,  studied in \cite{aiv1, russo, mt}. This  is the probability that two subgroups of $G$ commute, namely 
 \begin{equation}\label{sdg}
\mathrm{sd}(G)= \frac{|\{(X,Y)\in \mathrm{L}(G) \times \mathrm{L}(G) \ | \ XY=YX\}|}{{ |\mathrm{L}(G)|}^2}.
\end{equation}
If any two randomly chosen subgroups of $G$ commute, then $G$ is called \textit{quasihamiltonian}, and these groups were classified  since long time by Iwasawa (see \cite{rs}). Abelian groups are of course quasihamiltonian, but  the quaternion group $Q_8$ of order 8 is a nonabelian group of $\mathrm{sd}(Q_8)=1$. Evidently $G$ is quasihamiltonian if and only if $\mathrm{sd}(G)=1$, therefore \eqref{sdg} is a measure of how far is a group from being quasihamiltonian. It will be useful to introduce the following sets 
\begin{equation}\mathcal{H}(G)=\{H \in \mathrm{L}(G) \mid \mathrm{sd}(H) \neq 1\}
 \ \ \mathrm{and} \ \ \mathcal{K}(G)=\{K \in \mathrm{L}(G) \mid \mathrm{sd}(K) = 1\}\end{equation}
which clearly determine a disjoint union of the form
\begin{equation}\label{extra4}
    \mathrm{L}(G)=\mathcal{H}(G) \cup \mathcal{K}(G).
\end{equation} Note that    permutable subgroups are subnormal, while normal subgroups are  of course permutable, see \cite{rs}. The combinatorial formulas, which were found in \cite[Theorem 1.3, Proposition 3.2, Corollary 3.3]{sr4}, illustrate important relations between \eqref{lm}, \eqref{lapl} and \eqref{sdg}. For instance, if \begin{equation}\mathrm{spec}(A(\Gamma_{\mathrm{L}(G)}))=\{\lambda_1,\lambda_2, \cdots, \lambda_m\}
 \ \ \mbox{and} \ \ \mathrm{spec}(L(\Gamma_{\mathrm{L}(G)}))=\{\sigma_1,\sigma_2, \cdots, \sigma_m\}\end{equation} are the spectrum of the adjacency and the Laplacian matrix respectively, then \cite[(3.6)]{sr4} shows that for groups with $\mathrm{sd}(G)\neq 1$ 
\begin{equation} \label{sdlam}
\mathrm{sd}(G)=1 - \frac{1}{{|\mathrm{L}(G)|}^2} \sum^m_{i=1} \lambda_i^2=1 - \frac{1}{{|\mathrm{L}(G)|}^2} \sum^m_{i=1} \sigma_i.
\end{equation}

Another important quantity which is associated to a group $G$ is the \textit{factorization number} 
\begin{equation}\label{f2}
F_2(G)=|\{(H,K)\in \mathrm{L}(G) \times \mathrm{L}(G) \ | \ G=HK\}|;
\end{equation}
 this denotes the number of all possible factorizations of $G$  as  product of two subgroups $H$ and $K$. In fact we say that a group $G$ has \textit{factorization} $HK$ if there are two subgroups $H$ and $K$ of $G$ such that $G=HK$ (see \cite{lp, farrokhi2}). 

We also mention from \cite[\S 1.1]{rs} that \textit{an interval}  of $\mathrm{L}(G)$ is the set \begin{equation}[K/H]=\{Z \in \mathrm{L}(G) \mid H \le Z \le K\},\end{equation} where $H \le K$. Note that $[K/H]$ is a sublattice of $\mathrm{L}(G)$. From \cite{rota}    the M\"obius function    $\mu : \mathrm{L}(G) \times \mathrm{L}(G)  \to \mathbb{Z} $ is recursively defined by: 
\begin{equation}\label{f2mob1}
 \sum_{Z\in [K/H]}  \mu(H, Z)=   \left\{\begin{array}{lcl}  1, &\,\,&  \ H=K,\\
\\
0, &\,\,& \mbox{otherwise}.
\end{array}\right. \,  
\end{equation}
In particular, the M\"obius number of $G$ is  $\mu(G)= \mu(1,G)$, considering $[G/1]=\mathrm{L}(G)$.

Our main result is the following: 

\begin{thm} \label{f2sp}
Let $G$ be a  group with $\mathrm{sd}(G) \neq 1$. Then 
\begin{equation}\label{extra1}F_2(G)= \Big(  \sum_{ K \in \mathcal{K}(G) }   \ |\mathrm{L}(K)|^2 \   \mu(K, G) \Big)  +  \Big(  \sum_{ H \in \mathcal{H}(G) } \Big ( \ |\mathrm{L}(H)|^2 \ -
\sum^m_{i=1} \sigma_i \ \Big )  \mu(H, G)\Big),
\end{equation}
where $m=|V(\Gamma_{\mathrm{L}(H)})|$ and  $\{\sigma_1,\sigma_2, \cdots, \sigma_m \} = \mathrm{spec}(L(\Gamma_{\mathrm{L}(H)}))$. In particular,  \begin{equation}\label{extra2}\mathrm{sd}(G)=\frac{1}{|\mathrm{L}(G)|^2}  
\Big(\sum_{S \in \mathrm{L}(G)}  \sum_{W \in \mathcal{K}(S)}   \ |\mathrm{L}(W)|^2 \   \mu(W, S)   + \sum_{S \in \mathrm{L}(G)} \sum_{U \in \mathcal{H}(S) }  \Big ( \ |\mathrm{L}(U)|^2 \ -
\sum^k_{j=1} \tau_j \ \Big )  \mu(U, S)\Big),
\end{equation}
where $k=|V(\Gamma_{\mathrm{L}(U)})|$  and $\{\tau_1,\tau_2, \cdots, \tau_k \} = \mathrm{spec}(L(\Gamma_{\mathrm{L}(U)}))$.
\end{thm}

We shall mention that    the theory of the subgroup commutativity degree has been recently discussed  in \cite{  laz1, sr,  russo,  farrokhi1, farrokhi2, mt}, but only  in \cite{sr2, sr4}  in connection with notions of spectral graph theory on the line of \cite{cd09, chung}. Therefore Theorem \ref{f2sp}  belongs to the line of research of \cite{sr2, sr4} and explores new connections with the theory of the factorization number in \cite{lp, farrokhi1, farrokhi2}. Section 2 collects  information of general nature on the references which are pertinent to the topic, but also some classical results on the partitions of groups. Section 3 contains the proof  of Theorem \ref{f2sp} along with some applications.

 \section{Groups with partitions, factorization number\\ and  subgroup commutativity degree}

In order to count the number of edges of the non-permutability graph of subgroups of a group $G$, combinatorial formulas were found in \cite[Lemma 2.10, Theorem 3.1]{sr2}  involving the subgroup commutativity degree. We report some results from \cite{sr2, sr4} below:

\begin{lem}[See \cite{sr4}, Lemma 2.5] \label{lemma 1}  
For a group $G$ we have
\begin{equation}\label{countingformula}
2 \ |E(\Gamma_{\mathrm{L}(G)})|= {|\mathrm{L}(G)|^2} \ (1- \mathrm{sd}(G)).
\end{equation}
\end{lem}

This formula shows that we can obtain the number of edges in $\Gamma_{\mathrm{L}(G)}$ if we know $\mathrm{sd}(G)$, and vice-versa. Moreover \cite[Proposition 3.2]{sr4} shows that $\mathrm{sd}(G)$ can be rewritten in terms of spectral invariants of $\Gamma_{\mathrm{L}(G)}$.

\begin{lem}[See \cite{sr4}, Theorem 1.2] \label{t:3.2} 
Let $G$  be a  group with $\mathrm{sd}(G) \neq 1$. Then $\mathrm{sd}(G)$ is invariant under the spectrum of  $A(\Gamma_{\mathrm{L}(G)})$. In particular, 
\begin{equation} \mathrm{sd}(G)=1-\frac{1}{|\mathrm{L}(G)|^2}\underset{X,Y\in V(\Gamma_{\mathrm{L}(G)})}{\overset{}{\sum}}a_{X,Y}.
\end{equation}
\end{lem}

The above formula  allows us to match an approach  of spectral nature with another of combinatorial nature (see \cite{aiv1,  mt16, laz1, farrokhi1}), since $\mathrm{sd}(G)$ may be obtained in terms of $F_2(G)$ by the formula
\begin{equation}\label{sdgfg}\mathrm{sd}(G)=\frac{1}{|\mathrm{L}(G)|^2}\sum_{H\in \mathrm{L}(G)} F_2(H).
\end{equation}
In fact \eqref{sdgfg} shows that the subgroup commutativity degree can be reduced to the computation of the factorization number. This  has led to important numerical evaluations for $\mathrm{sd}(G)$ via $F_2(H)$, because it was found that $F_2(H)$ may be expressed for several families of groups via Gaussian trinomial integers. Consequently, we may connect the spectral invariants of $\Gamma_{\mathrm{L}(G)}$ to $F_2(G)$ as indicated below.

 \begin{cor}[See \cite{sr4}, Lemma 2.6]
For a group $G$ we have  \begin{equation}\label{f2edg}
2 \ |E(\Gamma_{\mathrm{L}(G)})|= |\mathrm{L}(G)|^2 \ - \  \sum_{H\in \mathrm{L}(G)} F_2(H).
\end{equation}
 \end{cor}

Now we  report a few notions which are classical in the area of the theory of partitions of groups, referring mostly to \cite{baer1,    kegel2, ek,kont, zappa}.  

\begin{defn}[See \cite{ek}, Definition, \S 7.1]\label{hughes} Given a prime $p$ and a group $G$, \begin{equation}H_p(G)=\langle g \in G \mid g^p \neq 1\rangle
\end{equation} is the $Hughes$ $subgroup$ of  $G$.
\end{defn}

From Definition \ref{hughes}, $H_p(G)$ turns out to be the smallest subgroup of $G$ outside of which all elements of $G$ have order $p$. Of course, if $G$ has $\exp(G)=p$, then $H_p(G)=1$. Moreover $H_p(G)$ is a characterstic subgroup in $G$. The reader can refer to \cite[Chapter 7]{ek} for more information on Hughes subgroups and their role in the theory of groups with nontrivial partitions.

\begin{defn}[See \cite{zappa}, p.575]\label{hughes-thompson} A group $G$ is said to be a group of $Hughes$-$Thompson$ $type$ if it is not a $p$-group and $H_p(G) \neq G$ for some prime $p$. \end{defn}

It can be shown that groups as per Definition \ref{hughes-thompson} have  $H_p(G)$ nilpotent of $|G : H_p(G)| = p$, see \cite{kegel2}. Omitting details of the definitions,  we refer to \cite[Definition 8.1, Kapitel V, \S 8]{huppert}  for the notion of $Frobenius$ $group$, and to  \cite[Bemerkungen 10.15, 10.17, Kapitel II, \S 10]{huppert} for the notion of $Suzuki$ $group$ $\mathrm{Sz}(2^{2n+1})$.   Originally, Baer, Kegel and Kontorovich  \cite{baer1, kegel2,  kont, zappa} classified  groups with partitions, but the result below is due to Farrokhi:

\begin{thm}[See \cite{partitionsfarrokhi},  Classification Theorem, pp.119-120]\label{niceclassification}Let $G$ be a  group with
 a nontrivial partition. Then $G$ is isomorphic to exactly one of the following groups
\begin{itemize}
\item [(i).] $S_4$;
\item [(ii).] a $p$-group with $H_p(G) \neq G$;
\item [(iii).] a group of Hughes-Thompson type;
\item [(iv).] a Frobenius group;
\item [(v).]  $\mathrm{PSL}(2,p^n)$ for $p^n \ge 4  $;
\item [(vi).]  $\mathrm{PGL}(2,p^n)$ for $p^n \ge 5$ odd prime power;
\item [(vii).] $\mathrm{Sz}(2^{2n+1})$.
\end{itemize}
\end{thm}

We recalled Theorem \ref{niceclassification} here, because the subgroup commutativity degree has been computed for most of the groups with nontrivial partitions. Let's see this with more details. For instance, Farrokhi and Saeedi \cite{farrokhi1, farrokhi2} completely determined the factorization number of groups in Theorem \ref{niceclassification} (i), (v) and (vi). 

\begin{prop}[See \cite{farrokhi2}, Theorem 2.4]\label{fnpsl} The projective special linear group $\mathrm{PSL}(2,p^n)$ has
$$F_2(\mathrm{PSL}(2,p^n))=   \left\{\begin{array}{lcl}  2|\mathrm{L}(\mathrm{PSL}(2,p^n))| + 2 p ^n(p^{2n} - 1) - 1   &\,\,&  if \ \ p=2 \ and \ n>1,\\
\\
2|\mathrm{L}(\mathrm{PSL}(2,p^n))| +  p ^n(p^{2n} - 1) - 1   &\,&  if \ \ p>2, n>1, \ and \  (p^n -1)/2   \\ 
&\,\,& \ is \ odd, but \ p^n \neq 3, 7, 11, 19, 23, 59,\\
\\
2|\mathrm{L}(\mathrm{PSL}(2,p^n))|  - 1 &\,\,&  if  \ \ p>2, n>1, \ and \  (p^n -1)/2  \\ 
&\,\,& \ is \ even, but \ p^n \neq 5,9,29.
\end{array}\right. \,  
$$
In the other cases,  $$F_2(\mathrm{PSL}(2,p^n)) = 17, 27, 237, 1 141, 2 033, 4 935, 17 223, 48 261, 68 799, 780 695$$ if 
 $p^n = 2, 3, 5, 7, 9, 11, 19, 23, 29, 59$, respectively.
 \end{prop}
 
 Of course, one would like to evaluate numerically $|\mathrm{L}(\mathrm{PSL}(2,p^n))|$ in Proposition \ref{fnpsl} and this can be made in different ways. For instance, Shareshian \cite{jsha2} computed the M\"obius function \eqref{f2mob1} for $\mathrm{PSL}(2,p^n)$  and this helps to find $|\mathrm{L}(\mathrm{PSL}(2,p^n))|$. Another method is due to Dickson: we may list all the subgroups of $\mathrm{PSL}(2,p^n)$ and count them. Historically this was the first method  to investigate $|\mathrm{L}(\mathrm{PSL}(2,p^n))|$.

\begin{prop}[Dickson's Theorem, see \cite{huppert}, Hauptsatz 8.27, Kapitel II, \S 8]\label{subgroupspsl}\

The subgroups of $\mathrm{PSL}(2,p^n)$ are the following:
\begin{itemize}
\item[(i).] $p^n (p^n \pm 1)/2$ cyclic subgroups $C_d$ of order $d$, where $d$ is a divisor of $(p^n \pm 1)/2$;
\item[(ii).] $p^n (p^{2n} - 1)/(4d)$ dihedral subgroups $D_{2d}$ of order $2d$, where $d$ is a divisor of $(p^n \pm 1)/2$ and $d >2$ and $p^n (p^{2n} - 1)/24$ dihedral subgroups $D_4$;
\item[(iii).] $p^n (p^{2n} - 1)/24$ alternating subgroups $A_4$;
\item[(iv).]  $p^n (p^{2n} - 1)/24$ symmetric subgroups $S_4$ when $p^n \equiv 7 \mod 8$;
\item[(v).]  $p^n (p^{2n} - 1)/60$ alternating subgroups $A_5$ when $p^n \equiv \pm 1 \mod 10$;
\item[(vi).] $p^n  (p^{2n}  - 1)/(p^m  (p^{2m}  - 1))$  subgroups $\mathrm{PSL}(2,p^n)$ where $m$ is a divisor of $n$;
\item[(vii).]The elementary abelian group $C^m_p$ for $m \le n$;
 \item[(viii).]$C^m_p \rtimes C_d$, where $d$ divides both $(p^n - 1)/2$ and $p^m-1$.
\end{itemize}
\end{prop} 
 
A  result, which is similar to Proposition \ref{fnpsl}, is available for projective general linear groups. 

\begin{prop}[See \cite{farrokhi2}, Theorem 2.5]\label{fnpgl} For any $p >2$ let $M$ be the unique subgroup of $G=\mathrm{PGL}(2,p^n)$ isomorphic to $\mathrm{PSL}(2,p^n)$. If $p^n >29$, then
$$F_2(G)=   \left\{\begin{array}{lcl} 3p^n (p^{2n} - 1) + 4|L(G)| - 2|L(M)| - 3  &\,  if \ n \   even \  or \   p \equiv 1 \pmod{4},\\
\\
4p^n (p^{2n} - 1) + 4|L(G)| - 2|L(M)| - 3, \,& \ if \  n \   odd  \ and  \ p \equiv 3 \pmod{4}. \\ 
\end{array}\right. \,  
$$
In the other cases,  $$F_2(G) = 177, 1103, 3083, 4919, 15549, 14529, 31093, 58429, 111567, 99527,
144297, 192349$$ if 
 $p^n =  3, 5, 7, 9, 11, 13, 17, 19, 23, 25, 27, 29$, respectively.
 \end{prop}
 
Essentially, we may  compute the factorization number for all the groups which are mentioned in Theorem \ref{niceclassification}, referring to methods of combinatorics and number theory  in \cite{aiv1, aiv2,  farrokhi1, farrokhi2}, but let's focus only on $\mathrm{PSL}(2,p^n)$ and $\mathrm{PGL}(2,p^n)$, in order to show significant applications of the spectral invariants which we associated to $\Gamma_{\mathrm{L}(G)}$.

From Propositions \ref{fnpsl} and \ref{fnpgl},  a precise computation of the factorization number should involve a numerical evaluation of the cardinalities of the subgroups lattices. There are details again in \cite{farrokhi1, farrokhi2} in this sense and the main idea is to introduce the M\"obius function    \eqref{f2mob1}, as originally made by Hall \cite{hall2}. The case of $p$-groups is known since long time:
 
\begin{lem}[See \cite{hall}] \label{hall}
In a $p$-group $G$ of order $p^n$ we have $\mu(G)=0$, unless
$G$ is elementary abelian, in which case we have $\mu (G)= (-1)^n p^{n\choose 2}$.
\end{lem}

 In  case of a symmetric group,  $\mu(1, S_n)$ was compute by Shareshian \cite{jsha} and Pahlings \cite{pa}.

\begin{prop}[See \cite{jsha}, Theorems 1.6, 1.8, 1.10] \label{sharesian}\
\begin{enumerate}
\item[{\rm (i).}] Let $p$ be a prime. Then $\mu(1, S_p)=(-1)^{p-1}\frac{p!}{2}$.
\item[{\rm (ii).}] $ \mu(1, S_n)=\left\{\begin{array}{lcl}  -n!, &\,\,&  \ \mbox{if n-1 is prime and p=3 mod 4},\\
\\
\ \frac{n!}{2}, &\,\,&  \ \mbox{if n=22},\\
\\
\frac{-n!}{2}, &\,\,& otherwise,\end{array}\right. \, $
\item[{\rm (iii).}] Let $n=2^{\alpha}$ for an integer $\alpha \geq 1$. Then $\mu(1, S_n)=\frac{-p!}{2}$.
\end{enumerate}
\end{prop}

In addition to symmetric groups, Shareshian \cite{jsha2} computed $\mu (1,G)$ also for  projective general linear groups, projective special linear groups and for Suzuki groups, see \cite{jsha, jsha2}.

\section{Proof of the main theorem and some applications}

Our main result connects the factorization number of a group  with the spectrum of the Laplacian  matrix via the M\"obius  function.

\begin{proof}[Proof of Theorem \ref{f2sp}]
In a group $G$ we have always that
\begin{equation}\label{f2mob}F_2(G)=   \sum_{T \in \mathrm{L}(G)} \mathrm{sd}(T) \ |\mathrm{L}(T)|^2 \ \mu(T, G)
\end{equation}
This is just an application of the  M\"obius  Inversion Formula to \eqref{sdgfg}.

Note from \cite{sr2} that $\Gamma_{\mathrm{L}(G)}$ is a null graph whenever  $G$ is quasihamiltonian. Then, in what follows, we shall  assume that $G$ is not quasihamiltonian and $K$ is an arbitrary subgroup of $G$ of $\mathrm{sd}(K)=1$. Consequently,  $\Gamma_{\mathrm{L}(K)}$ is the null graph. Similarly, we assume $H$ to be an arbitrary subgroup of $G$ of $\mathrm{sd}(H)\neq 1$. Consequently, $\Gamma_{\mathrm{L}(H)}$ exists and is different from the null graph.  From   Lemma \ref{t:3.2}, we have for $m_T=|V(\Gamma_{\mathrm{L}(T)})|$
\begin{equation}\label{extra3}\mathrm{sd}(T)=1 - \frac{1}{{|\mathrm{L}(T)|}^2} \sum^{m_T}_{i=1} \sigma_i .
\end{equation}
and so we can use \eqref{f2mob},  obtaining    
 \begin{equation}\label{technical}F_2(G)=  \sum_{T \in \mathrm{L}(G)} \Big ( \ |\mathrm{L}(T)|^2 \ -
\sum^{m_T}_{i=1} \sigma_i \ \Big )  \mu(T, G).
\end{equation}
But if $T \in \mathcal{K}(G)$ in \eqref{extra4}, then  $\Gamma_{\mathrm{L}(K)}$ is the null graph and so  we may assume each $\sigma_i=0$  with respect to  $L(\Gamma_{\mathrm{L}(K)})$. Hence we get
\begin{equation}\label{technicalbis}
F_2(G)= \sum_{K \in \mathcal{K}(G)} \Big ( \ |\mathrm{L}(K)|^2 \ -
\sum^{m_K}_{i=1} \sigma_i \ \Big )  \mu(K, G)  + \sum_{H \in \mathcal{H}(G)} \Big ( \ |\mathrm{L}(H)|^2 \ -
\sum^{m_H}_{i=1} \sigma_i \ \Big )  \mu(H, G) \end{equation}
\[=\sum_{K \in \mathcal{K}(G)} \Big ( \ |\mathrm{L}(K)|^2    \mu(K, G) \Big) + \sum_{H \in \mathcal{H}(G)} \Big ( \ |\mathrm{L}(H)|^2 \ -
\sum^{m_H}_{i=1} \sigma_i \ \Big )  \mu(H, G),\]
where $m_H=m=|V(\Gamma_{\mathrm{L}(H)})|$ as claimed.

From \eqref{sdgfg} and \eqref{technicalbis}, now we consider an arbitrary   $S \in \mathrm{L}(G)$ and a corresponding partition  $\mathrm{L}(S)=\mathcal{H}(S) \cup \mathcal{K}(S)$, as made for $G$ in \eqref{extra4}. We get

\begin{equation}\label{technicaltris}
|\mathrm{L}(G)|^2 \ \mathrm{sd}(G)=  \sum_{S \in \mathrm{L}(G)}F_2(S) \end{equation} 
\[=\sum_{S \in \mathrm{L}(G)}\Big( \sum_{W \in \mathcal{K}(S)}  \ |\mathrm{L}(W)|^2 \   \mu(W, S)   +   \sum_{U \in \mathcal{H}(S)} \Big ( \ |\mathrm{L}(U)|^2 \ -
\sum^k_{j=1} \tau_j \ \Big )  \mu(U, S)\Big)
\]
$$=\sum_{S \in \mathrm{L}(G)} \sum_{W \in \mathcal{K}(S)}  \ |\mathrm{L}(W)|^2 \   \mu(W, S)   + \sum_{S \in \mathrm{L}(G)}  \sum_{U \in \mathcal{H}(S)} \Big ( \ |\mathrm{L}(U)|^2 \ -
\sum^k_{j=1} \tau_j \ \Big )  \mu(U, S)
$$  
in correspondence of $\{\tau_1,\tau_2, \cdots, \tau_k \} = \mathrm{spec}(L(\Gamma_{\mathrm{L}(U)}))$. The result follows. 
\end{proof}

Of course, we may repeat the proof of Theorem \ref{f2sp}, replacing \eqref{extra3} with the first equation in \eqref{sdlam} and involving $\mathrm{spec}(A(\Gamma_{\mathrm{L}(G)}))$ instead of $\mathrm{spec}(L(\Gamma_{\mathrm{L}(G)}))$. 

\begin{cor}\label{newconsquence}
Let $G$ be a  group with $\mathrm{sd}(G) \neq 1$. Then 
\begin{equation}\label{extra6}F_2(G)= \Big( \sum_{ K \in \mathcal{K}(G) }   \ |\mathrm{L}(K)|^2 \   \mu(K, G) \Big)  +  \Big( \sum_{ H \in \mathcal{H}(G) } \Big ( \ |\mathrm{L}(H)|^2 \ -
\sum^m_{i=1} \lambda^2_i \ \Big )  \mu(H, G)\Big),
\end{equation}
where $m=|V(\Gamma_{\mathrm{L}(H)})|$ and  $\{\lambda_1,\lambda_2, \cdots, \lambda_m \} = \mathrm{spec}(A(\Gamma_{\mathrm{L}(H)}))$. In particular,  \begin{equation}\label{extra7}\mathrm{sd}(G)=\frac{1}{|\mathrm{L}(G)|^2}  
\Big(\sum_{S \in \mathrm{L}(G)} \sum_{W \in \mathcal{K}(S)}   \ |\mathrm{L}(W)|^2 \   \mu(W, S)   + \sum_{S \in \mathrm{L}(G)} \sum_{U \in \mathcal{H}(S) } \ \Big ( \ |\mathrm{L}(U)|^2 \ -
\sum^k_{j=1} \rho^2_j \ \Big )  \mu(U, S)\Big),
\end{equation}
where $k=|V(\Gamma_{\mathrm{L}(U)})|$  and $\{\rho_1,\rho_2, \cdots, \rho_k \} = \mathrm{spec}(A(\Gamma_{\mathrm{L}(U)}))$.
\end{cor}

We present a few applications of Theorem \ref{f2sp}, but  some relevant comments should be made.

\begin{rem}\label{relevantcomment} Suppose   to compute $F_2(G)$  for  $G=\mathrm{PSL}(2,p^n)$. We may proceed as below:

\begin{itemize}

\item[(1).] Use Proposition \ref{fnpsl} and compute $|\mathrm{L}(G)|$ applying Proposition \ref{subgroupspsl}.

\item[(2).] Apply \eqref{extra1} of Theorem \ref{f2sp}, but in order to do this we should previously:
\begin{itemize}
\item[(a).]Determine $\Gamma_{\mathrm{L}(H)}$ and $\mathrm{spec}(L(\Gamma_{\mathrm{L}(H)}))$ in \eqref{extra1}; 
\item[(b).] Find the M\"obius numbers $\mu(H,G)$ and $\mu(K,G)$ in \eqref{extra1}.
\item[(c).] Find  $|\mathrm{L}(H)|$ and $|\mathrm{L}(K)|$ in \eqref{extra1}.
\end{itemize}
\end{itemize}
\end{rem}

The method (1) has been introduced in \cite[Lemma 3.2, Corollary 3.3]{farrokhi2}. The method (2) is presented here for the first time and is apparently harder than (1), but softwares are available such as GAP \cite{gap} and NewGraph \cite{software} which can assist better with the steps (2a), (2b) and (2c). Therefore it is very efficient. We  sketch  similar techniques for the corresponding subgroup commutativity degrees.

\begin{rem}\label{relevantcommentbis} Suppose   to compute $\mathrm{sd}(G)$  for  $G=\mathrm{PSL}(2,p^n)$. We may proceed as below:

\begin{itemize}

\item[(I).] Combine Propositions \ref{fnpsl} and \ref{subgroupspsl} for the computation of $F_2(H)$  where $H \in \mathrm{L}(G)$ with the formula \eqref{sdgfg}.

\item[(II).] Apply \eqref{extra2} of Theorem \ref{f2sp},  but in order to do this we should previously:
\begin{itemize}
\item[(a).]Determine $\Gamma_{\mathrm{L}(U)}$,  $L(\Gamma_{\mathrm{L}(U)})$ and $\mathrm{spec}(L(\Gamma_{\mathrm{L}(U)}))$ in \eqref{extra2}; 
\item[(b).] Find the M\"obius numbers $\mu(W,S)$ and $\mu(U,S)$ in \eqref{extra2}.
\item[(c).] Find  $|\mathrm{L}(U)|$ and $|\mathrm{L}(W)|$ in \eqref{extra2}.
\end{itemize}
\item[(III).] Apply \eqref{sdlam}, after computing $|\mathrm{L}(G)|$ and $\mathrm{spec}(L(\Gamma_{\mathrm{L}(G)}))$.
\end{itemize}
\end{rem}

The method (I) has been followed in \cite[Theorem 3.4]{farrokhi2}. The method (II) is presented here for the first time. The method (III) has been introduced in \cite{sr4}. The difference is subtle between (II) and (III): for small groups we prefer of course (III), but for large groups with big $\mathcal{K}(S)$ in \eqref{extra2} and small $\mathcal{H}(S)$ (or viceversa) (II) gives soon a qualitative evaluation of $\mathrm{sd}(G)$. For instance,  a $minimal$ $nonabelian$ $group$ $M$ is a group  which is nonabelian but all of whose proper subgroups are abelian. In this situation, one has $\mathcal{K}(M)=\mathrm{L}(M) \setminus \{M\}$ and $\mathcal{H}(M)= \{M\}$  from the  definitions. Then (II) is more convenient than (III) here. Note that minimal nonabelian groups were classified by Redei  \cite[Aufgabe 14, Kapitel III, \S 5 ]{huppert}.

The following examples illustrate Theorem \ref{f2sp} in the spirit of Remarks \ref{relevantcomment} and \ref{relevantcommentbis}.

 \begin{ex}\label{example1} The symmetric group $S_4$  is presented by $S_4=\langle a,b,c \ | \ a^2=b^3=c^4=abc=1 \rangle$, where $a =(12)$, $b=(123)$ and $c=(1234)$. It is well known that the set of all normal subgroups forms a sublattice of the subgroups lattice of a given group (see \cite{rs}). In other words, the set $\mathrm{N}(S_4)$ of all normal subgroups of $S_4$ is a sublattice of $\mathrm{L}(S_4)$ and we have \begin{equation}\mathrm{N}(S_4)=\{\{1\},\langle (12)(34),(13)(24) \rangle, A_4,S_4\} .
 \end{equation} Moreover, one can check that 
 
 \begin{equation}   \mathfrak{C}_{\mathrm{L}(S_4)}(\mathrm{L}(S_4))=\mathrm{N}(S_4),
 \end{equation}  
 since we have
$$\mathrm{L}(S_4)=\{\{1\}, \langle (12) \rangle, \langle (13)\rangle, \langle (23) \rangle, \langle (14) \rangle, \langle (24) \rangle, \langle (34) \rangle, \langle (13)(24) \rangle, \langle (14)(23) \rangle,\langle (12)(34) \rangle, $$
$$\langle (123) \rangle, \langle (124) \rangle, \langle (134) \rangle,  \langle (234) \rangle, \langle (1234) \rangle, \langle (1324) \rangle, \langle (1423) \rangle, \langle (12)(34),(13)(24) \rangle,\langle (13),(24) \rangle,$$
$$\langle (14),(23) \rangle, \langle (12),(34) \rangle,\langle (123),(12) \rangle,\langle (124),(12) \rangle, \langle (134),(13) \rangle,\langle (234),(23) \rangle,$$
\begin{equation}\langle (1234),(13) \rangle,\langle (1243),(14) \rangle, \langle (1324),(12) \rangle, A_4,S_4 \}.
\end{equation}

There are $30$ elements in $\mathrm{L}(S_4)$ and these   are  divided into $11$ conjugacy classes and $9$ isomorphism types. It is easy to check that there are in $\mathrm{L}(S_4)$
\begin{itemize}

\item[-]$9$ subgroups isomorphic to $C_2$; 

\item[-] $4$ subgroups isomorphic to $C_3$; 

\item[-] $3$ subgroups isomorphic to $C_4$; 

\item[-] $3$ subgroups isomorphic to $C_2\times C_2$; 

\item[-] $4$ subgroups isomorphic to $S_3$;

\item[-] $3$ subgroups isomorphic to $D_4$.  

\end{itemize}

In particular, we find that \begin{equation}|V(\Gamma_{\mathrm{L}(S_4)})|=| \mathrm{L}(S_4) \setminus \mathrm{N}(S_4)|=26.
\end{equation}

Now we are going to focus on special subgroups of $S_4$. First of all, consider $A_4$ and its non-permutability graph of subgroups $\Gamma_{\mathrm{L}(A_4)}$. We have  7 vertices, namely   \begin{equation}V(\Gamma_{\mathrm{L}(A_4)})=\{\langle (123)\rangle, \langle (124)\rangle, \langle (134)\rangle, \langle (234)\rangle,\langle (12)(34)\rangle, \langle (14)(23)\rangle, \langle (13)(24)\rangle \},
\end{equation} since \begin{equation}\mathfrak{C}_{\mathrm{L}(A_4)}(\mathrm{L}(A_4))= \mathrm{N}(A_4)=\{\{1\},\langle (12)(34),(13)(24) \rangle ,A_4\}
\end{equation}
and a corresponding computation of edges can be done via \cite{software}, obtaining the graph below.

\begin{center}
\begin{tikzpicture}
[scale=.75,auto=left,every node/.style={circle,fill=blue!20}]

  \node (7) at (-6, 7) {$\langle(13)(24) \rangle$};
  \node (8) at (0, 7) {$\langle (14)(23) \rangle$};
  \node (9) at (6, 7) {$\langle (12)(34) \rangle$};

  \node (10) at (-6, 3) {$\langle 123 \rangle$};
  \node (11) at (-2, 1) {$\langle 124 \rangle$};
  \node (12) at (3, 1) {$\langle 134 \rangle$};
  \node (13) at (6, 3) {$\langle 234 \rangle$};
  
   \draw (7)--(10);
  \draw (7)--(11);
  \draw (7)--(12);
  \draw (7)--(13);

   \draw (8)--(10);
  \draw (8)--(11);
  \draw (8)--(12);
  \draw (8)--(13);

  \draw (9)--(10);
  \draw (9)--(11);
  \draw (9)--(12);
  \draw (9)--(13);
  
   \draw (10)--(11);
  \draw (10)--(12);
  \draw (10)--(13);
  
   \draw (11)--(12);
  \draw (11)--(13);
  
  \draw (12)--(13);

\end{tikzpicture}
 
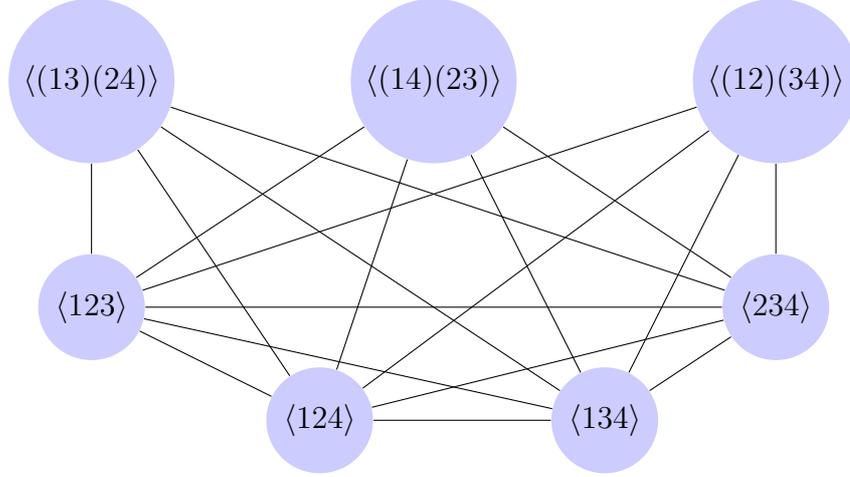
\textbf{ Figure 1}: The non-permutability graph of subgroups $\Gamma_{\mathrm{L} (A_{4})}$.
 \end{center}

\medskip
\medskip

Now we describe $B= \langle (123),(12) \rangle \simeq S_3  $ and   $\Gamma_{\mathrm{L}(B)}$. Here we get 
a triangle, because \begin{equation}V(\Gamma_{\mathrm{L}(B)})=\mathrm{L}(B) \setminus \mathfrak{C}_{\mathrm{L}(B)}(\mathrm{L}(B)) =  \mathrm{L}(B) \setminus \mathrm{N}(B)  =\{\langle (12)\rangle , \langle (13)\rangle, \langle (23)\rangle \}  
\end{equation}  and again \cite{software} can help with the computation of the edges. See below:

\vskip 0.4 true cm

\begin{center}
\begin{tikzpicture}
[scale=.75,auto=left,every node/.style={circle,fill=blue!20}]

  \node (1) at (-2, -2) {$\langle 12 \rangle$};
  \node (2) at (0, 0) {$\langle 13 \rangle$};
  \node (3) at (2, -2) {$\langle 23 \rangle$};

   \draw (1)--(2);
  \draw (2)--(3);
  \draw (3)--(1);

\end{tikzpicture}
 
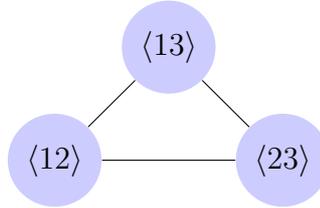
\textbf{ Figure 2}: The non-permutability graph of subgroups $\Gamma_{\mathrm{L} (B)}$ for $B \simeq S_3$.
 \end{center}

\medskip
\medskip

  Finally,  we consider $C= \langle (1234),(13) \rangle \simeq D_4 $ which has $\Gamma_{\mathrm{L}(C)}$  with 
 four vertices and four edges, namely
\begin{equation}V(\Gamma_{\mathrm{L}(C)})=\mathrm{L}(C) \setminus \mathfrak{C}_{\mathrm{L}(C)}(\mathrm{L}(C))=\{\langle (13)\rangle , \langle (24)\rangle, \langle (14)(23)\rangle, \langle (12)(34)\rangle \}.
\end{equation}
Again this is another very simple situation: the graph is a rectangle.

\vskip 0.4 true cm

\begin{center}
\begin{tikzpicture}
[scale=.75,auto=left,every node/.style={circle,fill=blue!20}]

  \node (1) at (-3, 0) {$\langle 13 \rangle$};
  \node (2) at (4, -3) {$\langle 24 \rangle$};
  \node (3) at (4, 0) {$\langle (14)(23) \rangle$};
  \node (4) at (-3, -3) {$\langle (12)(34) \rangle$};

   \draw (1)--(3);
  \draw (3)--(2);
  \draw (2)--(4);
  \draw (4)--(1);

\end{tikzpicture}
 
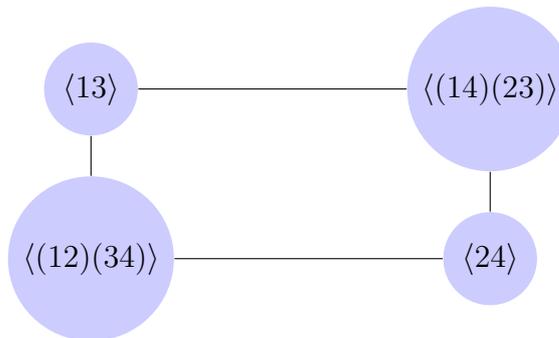
\textbf{ Figure 3}: The non-permutability graph of subgroups $\Gamma_{\mathrm{L} (C)}$ for $C \simeq D_4$.
 \end{center}

\medskip
\medskip

From Theorem \ref{f2sp}, we may  compute $F_2(S_4)$ in the following way:    \begin{equation}\label{computationsfour}
F_2(S_4)= \Big ( \sum_{ K \in \mathcal{K}(S_4) }  \ |\mathrm{L}(K)|^2 \   \mu(K, S_4) \Big )  +  \Big (  \sum_{ H \in \mathcal{H}(S_4) } \Big ( \ |\mathrm{L}(H)|^2 \ -
\sum^m_{i=1} \sigma_i \ \Big )  \mu(H, S_4)\Big ),
\end{equation}
where $K$ is a subgroup of $S_4$ belonging to
$$ \mathcal{K}(S_4)=\{ \{1\}, \langle 12 \rangle, \langle 13\rangle, \langle 23 \rangle, \langle 14 \rangle, \langle 24 \rangle, \langle 34 \rangle, \langle (13)(24) \rangle, \langle (14)(23) \rangle,\langle (12)(34) \rangle, \langle 123 \rangle, \langle 124 \rangle,$$ 
 \begin{equation}  \langle 134 \rangle,  \langle 234 \rangle, \langle 1234 \rangle, \langle 1324 \rangle, \langle 1423 \rangle, \langle (12)(34),(13)(24) \rangle,\langle (13),(24) \rangle,\langle (14),(23) \rangle,\langle (12),(34) \rangle   \},
 \end{equation}
and $H$ a subgroup of $S_4$ belonging to
$$\mathcal{H}(S_4)= \{\langle (123),(12) \rangle,\langle (124),(12) \rangle, \langle (134),(13) \rangle,\langle (234),(23) \rangle, $$ 
\begin{equation}
\langle (1234),(13) \rangle,\langle (1243),(14) \rangle, \langle (1324),(12) \rangle, A_4,S_4  \}. 
\end{equation}
Now we need to find $\mu(K, S_4)$ and $\mu(H, S_4)$ for all $K$ and $H$, but it is  enough  to find these values for each conjugacy classes only. Using Lemma \ref{hall} and Proposition \ref{sharesian} (iii), we find $$\mu(\{1\}, S_4)=-n!=-24,  \ \ \mu(\langle 12 \rangle, S_4)=2, \ \ \mu(\langle (13)(24) \rangle, S_4)=0, \ \ \mu(\langle 123 \rangle, S_4)=1, $$
$$\mu(\langle (12)(34),(13)(24) \rangle, S_4)=3, \ \  \mu(\langle (13),(24) \rangle, S_4)=0, \ \ \mu(\langle 1234 \rangle, S_4)=0, $$
\begin{equation}  \mu(\langle (123), (12) \rangle, S_4)=-1,  \ \ \mu(\langle (1234),(13) \rangle, S_4)=-1, \ \ \mu( A_4, S_4)=-1.   \ \ \mu( S_4, S_4)=1.\end{equation} On the other hand, we may use \cite{software}, in order to find the spectra of the Laplacian matrices $L(\Gamma_{\mathrm{L}(B)})$, $L(\Gamma_{\mathrm{L}(C)})$ and $L(\Gamma_{\mathrm{S}(A_4)})$, obtaining
 \begin{equation}\label{extra5}
 \mathrm{spec}(L(\Gamma_{\mathrm{L}(B)})) = \{0,3,3\}, \ \ \mathrm{spec}(L(\Gamma_{\mathrm{L}(C)}))=\{0,2,2,4\}, \ \ 
\mathrm{spec}(L(\Gamma_{\mathrm{L}(A_4)}))=\{0,4,4,7,7,7,7 \},
 \end{equation}  but we haven't reported all the details of the non-permutability graph $\Gamma_{\mathrm{L}(S_4)}$, since it is very technical. Just to give an idea, 
$$ \mathrm{spec}(L(\Gamma_{\mathrm{L}(S_4)})) =\{0,7.22863,7.60860,7.60860, 11.39978, 11.39978, 11.72495, 12.01650,$$$$ 12.01650,14, 14.56069, 14.56069, 14.56069, 15.61486, 16.33888, 16.33888, 16.33888, $$
 \begin{equation}\label{technicalcomputation}17.29890, 17.29890, 18, 20.10043, 20.10043, 20.10043, 20.43156, 20.67622, 20.67622\}
 \end{equation} is the spectrum of the Laplacian matrix  $L(\Gamma_{\mathrm{L}(S_4)})$.

Replacing the values which we found in \eqref{computationsfour}, we get 
$$F_2(S_4)=-24+6(2^2)(2)+3(2^2)(0)+4(2^2)(1)+(5^2)(3)+3(4^2)(0)+3(3^2)(0)+4(6^2-6)(-1)$$
  \begin{equation}
 +3(10^2-8)(-1)+(10^2-36)(-1)+(30^2-378)(1)=177.
 \end{equation}
Note also that
 $$\mu(\{1\}, A_4)=4,  \ \ \mu(\langle (13)(24) \rangle, A_4)=0, \ \  \mu(\langle (12)(34),(13)(24) \rangle, A_4)=-1,$$  \begin{equation}\mu(\langle (123) \rangle, A_4)=-1 , \ \   \mu( A_4, A_4)=1,
 \end{equation} imply with a similar argument that
  \begin{equation}F_2(A_4)=4+3(2^2)(0)+4(2^2)(-1)+(5^2)(-1)+(10^2-36)(1)=27.
  \end{equation}

With our new method of computation,  we have just seen that Theorem \ref{f2sp} shows an alternative method of computational nature for $F_2(
\mathrm{PGL}(2,3))$ and $F_2(
\mathrm{PSL}(2,3))$. In fact $\mathrm{PSL}(2,3)\simeq A_4$ and $ \mathrm{PGL}(2,3)\simeq S_4$, then  $F_2(\mathrm{PSL}(2,3))=F_2(A_4)=27$ and $F_2(\mathrm{PGL}(2,3))=F_2(S_4)=177$,  which are the  same values found  in Propositions \ref{fnpsl} and \ref{fnpgl}.
\end{ex}

Note that some  open problems were posed by Tarnauceanu \cite{mt} on the subgroup commutativity degree and the logic which we applied in Example \ref{example1}, along with Theorem \ref{f2sp} and \cite{software}, could bring solutions. In fact Remarks \ref{relevantcomment} and \ref{relevantcommentbis} suggest a methodology of general interest which can be applied to large families of groups, so not necessarily to linear groups.

We show another application of our main results.

\begin{ex}\label{example2} From a direct computation, if we consider $A_4$, then the denominator of \eqref{sdg} is equal to $100$, namely ${|\mathrm{L}(A_4)|}^2=100$ and the numerator of \eqref{sdg} is equal to $64$, hence 
  \begin{equation}\label{s4}
  \mathrm{sd}(A_4)=\frac{16}{25} 
  \end{equation}
according to \cite[p.2510]{mt}. On the other hand, we may consider  \eqref{extra5} and replace it in \eqref{extra3}
\begin{equation} \label{s4,2}
\mathrm{sd}(A_4) = 1 - \frac{\sigma_1 + \ldots + \sigma_7}{{|\mathrm{L}(A_4)|}^2} =1 - \frac{36}{100} = \frac{16}{25}. 
\end{equation}
Moreover, it is easy to check that $A_4$ is minimal nonabelian, then $\mathcal{K}(A_4)=\mathrm{L}(A_4) \setminus \{A_4\}$  and $\mathcal{H}(A_4)= \{A_4\}$. Now we can apply \eqref{extra1} to  obtain
 $F_2(\{1\})=1$, $F_2(\langle(13)(24)\rangle) =F_2(\langle (14)(23) \rangle)=F_2(\langle (12)(34) \rangle)=3$, 
 $F_2(\langle (123) \rangle)=F_2(\langle (124) \rangle)=F_2(\langle (13) \rangle)=F_2(\langle (234) \rangle)=3$, $F_2(\langle (12)(34),(13)(24) \rangle)=15$ and $F_2(A_4)=27$. Therefore, using\eqref{extra2} 
\begin{equation} \label{s4,3}
\mathrm{sd}(A_4) =  \frac{1+7(3)+15+27}{{|\mathrm{L}(A_4)|}^2} = \frac{16}{25}
\end{equation}
 which is the same value  obtained in \eqref{s4} and \eqref{s4,2} in different ways. 
 \end{ex}

Of course, we may repeat a similar arguments in Example \ref{example2}, in order to find $\mathrm{sd}(S_3)$, $\mathrm{sd}(S_4)$ and $\mathrm{sd}(D_4)$ on the basis of the values which we have in Example \ref{example1}, but we presented here just the case of $A_4$ supporting Remark \ref{relevantcommentbis} (III) and (II).



We end with the following problem, which we encountered in our investigations:

\begin{problem}Study systematically the non-permutability graph of subgroups for the groups in Theorem \ref{niceclassification}, developing a corresponding spectral graph theory for non-permutability graph of subgroups of groups with nontrivial partitions. Determine the subgroup commutativity degree of all the groups in Theorem \ref{niceclassification} via spectra of Laplacian matrices of the corresponding non-permutability graph of subgroups.
\end{problem}

\end{document}